\documentclass[12pt, reqno, a4paper]{amsart}

\usepackage{amsmath, amssymb, amsfonts, enumerate}
\usepackage{amsthm}
\usepackage{latexsym, xcolor, epsfig}
\usepackage{hyperref, mathrsfs}
\usepackage{esint}

\usepackage[numbers, square]{natbib}

\usepackage{scalerel}

\newcommand{\F}{\mathbb{F}}

\usepackage{dsfont}
\newcommand{\nd}{\text{ and }}

\newcommand{\Span}{\mathrm{Span}}

\renewcommand{\epsilon}{\varepsilon} % For autocompletion purposes
\renewcommand{\leq}{\leqslant} % Jury's out on these two
\renewcommand{\geq}{\geqslant}
\newtheorem{theorem}{Theorem}[section]
\newtheorem*{theorem*}{Theorem}
\newtheorem{conjecture}[theorem]{Conjecture}
\newtheorem{proposition}[theorem]{Proposition}
\newtheorem*{proposition*}{Proposition}

\newtheorem*{lemma*}{Lemma}

\newtheorem*{corollary*}{Corollary}

\newtheorem*{example*}{Example}

\newtheorem*{claim*}{Claim}

\theoremstyle{definition}

\newtheorem*{definition*}{Definition}

\theoremstyle{remark}

\newtheorem*{remark*}{Remark}

\newtheorem*{question*}{Question}

\numberwithin{equation}{section}

\title{A counterexample to a strong variant of the Polynomial Freiman-Ruzsa conjecture}
\author{James Aaronson}
\email{james.aaronson@maths.ox.ac.uk}

\date{}
\begin{document}

\begin{abstract}
Let $p$ be a prime. One formulation of the Polynomial Freiman-Ruzsa conjecture over $\mathbb{F}_p$ can be stated as follows. If $\phi : \mathbb{F}_p^n  \rightarrow \mathbb{F}_p^N$ is a function such that $\phi(x+y) - \phi(x) - \phi(y)$ takes values in some set $S$, then there is a linear map $\tilde{\phi} : \mathbb{F}_p^n  \rightarrow \mathbb{F}_p^N$ with the property that $\phi - \tilde{\phi}$ takes at most $|S|^{O(1)}$ values.

A strong variant of this conjecture states that, in fact, there is a linear map $\tilde{\phi}$ such that $\phi - \tilde{\phi}$ takes values in $tS$ for some constant $t$. In this note, we discuss a counterexample to this conjecture.
\end{abstract}

\maketitle

\section{Introduction}\label{sec:intro}

Let $A \subseteq \F_p^n$ be a subset, with $|A+A| \leq K|A|$. The Polynomial Freiman-Ruzsa conjecture, attributed to Marton, asserts that $A$ can be covered by at most $K^{O(1)}$ translates of a subspace of size at most $K^{O(1)} |A|$. 

An equivalent form of the Polynomial Freiman-Ruzsa conjecture may be stated as follows \cite{ben_article}:
\begin{conjecture}\label{conj:PFR}
Suppose that $\phi : \mathbb{F}_p^n  \rightarrow \mathbb{F}_p^N$ is a function such that $\phi(x+y) - \phi(x) - \phi(y)$ takes values in some set $S$. 

Then there is a linear map $\tilde{\phi} : \mathbb{F}_p^n  \rightarrow \mathbb{F}_p^N$ with the property that $\phi - \tilde{\phi}$ takes at most $|S|^{O(1)}$ values.
\end{conjecture}

It is tempting to consider the following strong form of Conjecture \ref{conj:PFR}:
\begin{conjecture}\label{conj:strong_PFR}
Suppose that $\phi : \mathbb{F}_p^n  \rightarrow \mathbb{F}_p^N$ is a function such that $\phi(x+y) - \phi(x) - \phi(y)$ takes values in some set $S$. 

Then there is a linear map $\tilde{\phi} : \mathbb{F}_p^n  \rightarrow \mathbb{F}_p^N$ with the property that $\phi - \tilde{\phi}$ takes values in $tS$, for some constant $t$ depending only on $K$.
\end{conjecture}

In \cite[Theorem 3.3]{farah_paper}, Farah gives an example which refutes Conjecture \ref{conj:strong_PFR}. Green and Tao give another example in \cite[Section 1.17]{ben_tao_example} for the case $p = 2$. However, their example relies crucially on the fact that $p = 2$, whereas Farah's example does not require this.

% https://terrytao.files.wordpress.com/2009/01/whatsnew.pdf

In this note, we offer an alternative approach to that of Farah. In particular, our proof gives quite reasonable quantitative bounds.

\section{The construction}\label{sec:construction}

The construction goes as follows.
\begin{theorem}\label{thm:farah_constr}
Given $t$, suppose that $n \geq 12t + 7$. View $\F_p^N$ as the space of all (not necessarily linear) maps from $\F_p^n$ to $\F_p$ (in particular, $N = p^{p^n}$).

Let $\phi : \F_p^n \rightarrow \F_p^N$ be defined as follows. If $v$ is an element of $\F_p^n$, then define $\phi(v) : \F_p^{\F_p^n} \rightarrow \F_p$ by
\[
\phi(v) : f \mapsto f(v),
\]
where $f : \F_p^n \rightarrow \F_p$.

Let $S \subseteq \F_p^N$ denote the set of values taken by $\phi(x+y) - \phi(x) - \phi(y)$. Then, there is no linear map $\tilde{\phi}$ such that $\phi - \tilde{\phi}$ takes values in $tS$.
\end{theorem}

\begin{proof}
Suppose that there does exist $\tilde{\phi}$ such that $\phi - \tilde{\phi}$ takes values in $tS$.

An element of $S$ may be written $\phi(a+b) - \phi(a) - \phi(b)$ for some $a, b \in \F_p^n$. Thus, for each $x \in \F_p^n$, there are pairs $(a_i^{(x)}, b_i^{(x)})$ for $1 \leq i \leq t$ with the property that
\begin{equation}\label{eqn:phi_decomp}
(\phi(x) - \tilde{\phi}(x))(f) = \sum_{i=1}^t f(a_i^{(x)} + b_i^{(x)}) - f(a_i^{(x)}) - f(b_i^{(x)}) 
\end{equation}
for every map $f : \F_p^n \rightarrow \F_p.$

For each $x \in \F_p^n$, define $V_x = \Span(x, a_1^{(x)}, b_1^{(x)}, \dots, a_t^{(x)}, b_t^{(x)})$. $V_x$ obeys the following three properties:
\begin{enumerate}
    \item $V_x$ has dimension at most $2t + 1$.
    
    \item $x \in V_x$.
    
    \item Suppose that $f : \F_p^n \rightarrow \F_p$ is a map such that $\phi(x)(f) \neq \tilde{\phi}(x)(f)$. Then the restriction $f|_{V_x}$ is nonlinear.
\end{enumerate}
The first two properties follow trivially. To see why the third holds, suppose that $\phi(x)(f) \neq \tilde{\phi}(x)(f)$. By (\ref{eqn:phi_decomp}), we learn that
\[
\sum_{i=1}^t f(a_i^{(x)} + b_i^{(x)}) - f(a_i^{(x)}) - f(b_i^{(x)}) \neq 0.
\]
Thus, at least one term in the sum must be zero; from that term, we have that 
\[ f(a_i^{(x)} + b_i^{(x)}) \neq f(a_i^{(x)}) + f(b_i^{(x)}). \]
However, all three of the arguments to $f$ above are in $V_x$, so $f|_{V_x}$ must be nonlinear.

Now, in order to find a contradiction, observe that it suffices to find a pair $x, y \in \F_p^n$ such that 
\begin{equation}\label{eqn:condition_on_xy}
x+y \notin (V_x \cap V_{x+y}) + (V_y \cap V_{x+y}).    
\end{equation}
To see why, construct a function $f : \F_p^n \rightarrow \F_p$ as follows: 
\begin{itemize}
    \item Set $f|_{V_x} = 0$ and $f|_{V_y} = 0$.
    
    \item Set $f(x+y) = 1$, and extend to a linear function on $V_{x+y}$. This is possible because of our condition on $x + y$.
    
    \item Define $f$ arbitrarily on $\F_p^n \backslash (V_x \cup V_y \cup V_{x+y})$.
\end{itemize}
Now, we have that
\begin{align*}
    0 = f(x) &= \phi(x)(f) = \tilde{\phi}(x)(f) \\
    0 = f(y) &= \phi(y)(f) = \tilde{\phi}(y)(f) \\
    1 = f(x+y) &= \phi(x+y)(f) = \tilde{\phi}(x+y)(f).
\end{align*}
In each line, the first equality holds by construction, the second follows from the definition of $\phi$ and the third follows from property (3) since $f|_{V_x}, f|_{V_y} \nd f|_{V_{x+y}}$ are linear. Thus $\tilde{\phi}(x+y) - \tilde{\phi}(x) - \tilde{\phi}(y)$ is not zero when evaluated at $f$, contradicting the linearity of $\tilde{\phi}$.

Thus, we will be done if we can establish the following:
\begin{proposition}
Suppose that for each $x \in \F_p^n$, we have a subspace $V_x$ of dimension at most $2t + 1$. Then, provided that $n \geq 12t + 7$, there must exist some pair $x, y$ satisfying (\ref{eqn:condition_on_xy}). 
\end{proposition}

\begin{proof}
Suppose that no such pair $x, y$ exist. Then, for each $x, z$ there exist $v_{x, y} \in V_x \cap V_z$ and $w_{x,z} \in V_{z-x} \cap V_{z}$, with $z = v_{x,z} + w_{x,z}$.

As $x$ runs over $\F_p^n$, $v_{x,z}$ takes values in $V_z$. Thus, by the pigeonhole principle, there must exist some choice $v_z \in V_z$ which occurs for at least $p^{n - 4t - 2}$ choices of $x$. 

If we define $U = \{v \in \F_p^n | v \in V_x \text{ for at least } p^{n - 4t - 2} \text{ choices of } x\}$, then we instantly learn that $v_z \in U$ for each $z \in \F_p^n$, since $v_z \in V_x$ for at least $p^{n - 4t - 2}$ choices of $x$. Similarly, $w_z \in U$, since $w_z \in V_{z - x}$. In view of the fact that $v_z + w_z = z$, we learn that $U + U$ is the whole of $\F_p^n$, and so $|U| \geq p^{n/2}$.

However, we can also give an upper bound for $|U|$. There are at most $p^{n + 2t + 1}$ pairs $v, x$ with $v \in V_x$, and each $v \in U$ must count at least $p^{n - 4t - 2}$ of those pairs. Thus, there can be at most $p^{6t + 3}$ elements of $U$.

Putting this together, we have that $p^{n/2} \leq |U| \leq p^{6t + 3}$, and so that $n \leq 12t + 6$.
\end{proof}
This gives us the required contradiction, and so such a linear function $\tilde{\phi}$ indeed cannot exist.
\end{proof}

\addtocontents{toc}{\protect\vspace*{\baselineskip}}
%\addcontentsline{toc}{section}{Appendices}

%\clearpage

\addtocontents{toc}{\protect\vspace*{\baselineskip}}

\addcontentsline{toc}{section}{References}

\bibliographystyle{plainnat}

\end{document}